\documentclass[runningheads]{llncs}
\usepackage[T1]{fontenc}
\usepackage{graphicx}

\usepackage{microtype}
\usepackage{subfigure}
\usepackage{booktabs} 
\usepackage{longtable}
\usepackage[table]{xcolor}
\usepackage{multirow}
\usepackage{todonotes}
\usepackage{adjustbox} 
\usepackage{pdflscape}
\usepackage{siunitx}
\usepackage{hyperref}

\usepackage{amsmath}
\usepackage{amssymb}
\usepackage{mathtools}

\usepackage[capitalize,noabbrev]{cleveref}
\usepackage{dsfont}
\usepackage[numbers]{natbib}
\usepackage{algorithm}
\usepackage{algorithmic}
\usepackage{enumerate}

\hypersetup{breaklinks,colorlinks,citecolor=blue!50!black,linkcolor=blue!60!black,urlcolor=blue!60!black}

\newcommand{\ceil}[1]{\lceil #1 \rceil}

\newcommand{\abs}[1]{\lvert{#1}\rvert}

\newcommand{\Z}{\mathds{Z}}
\newcommand{\R}{\mathds{R}}
\newcommand{\define}{\coloneqq}
\newcommand{\T}{^\top}
\newcommand{\suchthat}{\,:\,}
\newcommand{\F}{\mathcal{F}}

\begin{document}
\title{Objective Coefficient Rounding and Almost Symmetries in Binary Programs}

\author{
Dominik Kuzinowicz\inst{1}\orcidID{0009-0005-1156-106X} \and
Paweł Lichocki\inst{4}\orcidID{0009-0002-4422-0876}\and
Gioni Mexi\inst{1}\orcidID{0000-0003-0964-9802} \and
Marc E. Pfetsch\inst{3}\orcidID{0000-0002-0947-7193} \and
Sebastian Pokutta\inst{1,2}\orcidID{0000-0001-7365-3000} \and
Max Zimmer\inst{1}\orcidID{0009-0007-8683-1030}
}

\authorrunning{D. Kuzinowicz et al.}

\institute{
  Zuse Institute Berlin,
  \email{\{kuzinowicz, mexi, pokutta, zimmer\}@zib.de}
  \and Technische Universität Berlin
  \and
  Google,
  \email{pawell@google.com}  
  \and
  Department of Mathematics, TU Darmstadt,
  \email{pfetsch@mathematik.tu-darmstadt.de}
}

\maketitle              
\begin{abstract}
   This article investigates the interplay of rounding objective coefficients in binary programs and almost symmetries.
   Empirically, reducing the number of significant bits through rounding often leads to instances that are easier to solve.
   One reason can be that the amount of symmetries increases, which enables solvers to be more effective when they are exploited.
   This can signify that the original instance contains `almost symmetries'.
   Furthermore, solving the rounded problems provides approximations to the original objective values.
   We empirically investigate these relations on instances of the capacitated facility location problem, the knapsack problem and a diverse collection of additional instances, using the solvers SCIP and CP-SAT.
   For all investigated problem classes, we show empirically that this yields faster algorithms with guaranteed solution quality. The influence of symmetry depends on the instance type and solver.
   \keywords{symmetry handling \and $\ell$-bit precision}
\end{abstract}
\section{Introduction}

In this article, we investigate the connection between two fundamental concepts in integer programming (IP) and constraint programming (CP): the rounding of the objective coefficients to a prescribed number of bits and almost symmetries.
Both have a significant influence on the performance of solution algorithms.
To explain these concepts, consider a \emph{Binary Program} (BP)
\begin{equation}\label{eq:Prob}
\max\; \{c\T x \suchthat x \in \F\},
\end{equation}
where $\F \subseteq \{0,1\}^n$ and we assume integral $c \in \Z^n$.
Then \emph{rounding} $c$ with respect to a given number of bits~$\ell$ (\emph{$\ell$-bit rounding}) means to set all bits except the top $\ell$ of each of its components to zero.
At one extreme, $\ell = 0$ allows no bits, 
yielding the zero objective and turning the BP into a feasibility problem, which is typically much easier to solve in practice.
The other extreme uses the original number of significant bits $\ell = \ceil{\log_2(\max\{\abs{c_1} + 1, \dots, \abs{c_n} + 1\})}$ and thus leaves the objective unchanged.
Empirically, it can be observed that reducing the number of bits~$\ell$ often leads to easier instances.

The second concept, `almost symmetries', has appeared in the literature under different meanings.
It usually refers to the situation that in a branch-and-bound tree some subproblems at the nodes behave similarly, i.e., almost symmetrical, and thus resources are wasted in order to investigate very similar solutions without gaining relevant new information. In this paper, we propose the following viewpoint: The BP contains \emph{almost symmetries} if the amount of symmetries increases when decreasing the number of bits~$\ell$ in rounding. This also provides a measure for being almost symmetric. Here, a \emph{symmetry} of~\eqref{eq:Prob} is a permutation~$\sigma$ of the variables such that $x \in \F$ if and only if $\sigma(x) \in \F$ and $c\T x = c\T \sigma(x)$.

One `classical' example is the well-known FPTAS (Fully Polynomial Time Approximation Scheme) for the knapsack problem, i.e., the special case of BP with $\F \define \{x \in \{0,1\}^n \suchthat w\T x \leq W\}$ for $w \in \Z_+^n$, $W \in \Z_+$. The basic idea is to round the objective coefficients to a multiple of $\varepsilon > 0$. In this way, the number of different coefficients appearing in a dynamic programming table reduces, which is the main factor determining table sizes and running times. By rounding such that the number of different coefficients become polynomial, one can obtain an approximation of the true optimum with accuracy $\varepsilon$, see, e.g., \cite[Chapter~17]{KorV18}.

\noindent
\textbf{Contributions.}
In this paper, we illustrate the effect of $\ell$-bit rounding on symmetry and solving time using the capacitated facility location and knapsack problem, as well as collected instances from various sources. The performance of the rounded version of each instance is evaluated using two solvers, namely SCIP~\cite{BestuzhevaEtAl23,BolusaniEtal2024OO} and CP-SAT~\cite{ortools_cp} CP-Solver. We demonstrate that indeed with decreasing number of significant bits, i.e., smaller~$\ell$, the problems become easier to solve, while the amount of symmetry increases. We approximate the amount of symmetries by the number of generators of the symmetry group.
Moreover, solving the problem with $\ell$-bit rounding will give an $\frac{1}{2^{\ell-1}}$-approximation of the original problem (Lemma~\ref{lemma:Approximation}).
In this article, we show empirically that depending on the instance type and solver, increasing the amount of symmetries via rounding can help symmetry handling methods to be more successful and allow to reduce the solving time.
All this is possible with black-box solvers by only changing the objective coefficients.

\noindent
\textbf{Literature.}
The concepts of almost symmetries and bit rounding have been studied in the literature.
The term `almost symmetries' has no unique definition, but it has been investigated in several works; see, for example, the proceedings of the workshop ``Almost-Symmetry in Search''~\cite{DonG05}. Moreover, \cite{KnuOP18} investigate when graphs are almost symmetric, which means that they can become symmetric by applying a small number of graph operations.

Bit rounding plays an important role in the investigation of the equivalence of optimization, separation, and augmentation. For example, \cite{schulz19950} and \cite{schulz2002complexity} proved the equivalence of these problems using an iterated rounding scheme called bit-scaling and variants thereof. See~\cite{BPPP2015} for a practical evaluation of these algorithms. 

Interestingly, changing the objective to make optimal solutions unique has also been used to `break symmetries', see e.g., \cite{GhoS11}. This usually requires to increase the number of significant bits and hinders symmetry handling methods.
Moreover, there are many techniques to handle symmetries in IP and CP, see~\cite{Mar10} for an overview, \cite{PfeR19} for a computational comparison, and \cite{BolusaniEtal2024OO} for a presentation of the state-of-the-art techniques used in SCIP.
We note that computing symmetries as defined above is NP-hard, but so-called \emph{formulation symmetries} that leave the linear formulation invariant can be computed using graph automorphism algorithms, see, e.g., \cite{PfeR19}.
\smallskip

\noindent
\textbf{Outline.}
In Section \ref{sec:Basics} we discuss $\ell$-bit rounding, its connection to $\varepsilon$-optimality in optimization and prove a worst-case bound induced by $\ell$-bit rounding. In Section~\ref{sec:experiments} we explain our experimental setup, introduce our problem classes and analyze our results. Finally, in Section \ref{sec:conclusion} we conclude our findings.

\section{$\ell$-bit rounding and $\varepsilon$-optimality}
\label{sec:Basics}

The rounding of the objective $c = (c_1, \ldots, c_n)\T \in \Z^n$ is defined as follows.
Consider the binary representation $c_i = \text{sgn}(c_i) \sum_{j = 0}^{k_i-1} c_{ij}\, 2^j$ of $c_i$, where $\text{sgn}(c_i) \in \{0,\pm 1\}$ is the sign of $c_i$, $k_i = \ceil{\log_2(\abs{c_i} + 1)}$, $c_{ij} \in \{0,1\}$ for all $j = 0, \dots, k_i - 1$.
For $\ell \in \Z_+$, the \emph{$\ell$-bit rounding} of $c_i$ is then defined as $\text{sgn}(c_i) \sum_{j = k_i - \ell}^{k_i - 1} c_{ij}\, 2^j$, that is, setting the lowest $k_i - \ell$ bits to zero and leaving the remaining unchanged.
To quantify the loss in precision that we get by $\ell$-bit rounding, consider the following definition.

\begin{definition}[see \cite{orlin2008}]\label{def:epsilon_optimal}
    A feasible $x^\star$ is $\varepsilon$-optimal for the objective function $c$ if there exists $c' \in \R^n$ such that
    \begin{align*}\label{eq:perturbed_optimality}
    (1-\varepsilon)\, c_j & \leq c_j' \leq (1+\varepsilon)\, c_j && \text{if } c_j \geq 0, \\
    -(1-\varepsilon)\, c_j & \leq c_j' \leq -(1+\varepsilon)\, c_j && \text{if } c_j < 0
    \end{align*}
    for all $j \in [n] \define \{1, \dots, n\}$, and $x^\star$ is optimal for $c'$.
\end{definition}
We prove the following Lemma$\colon$
\begin{lemma}\label{lemma:Approximation}
   Let $x^\star$ be an optimal solution for the objective $c'$ obtained by $\ell$-bit rounding. Then $x^\star$ is $\frac{1}{2^{\ell-1}}$-optimal with respect to the original objective.
\end{lemma}

\begin{proof}
    Consider $i \in [n]$ with $c_i \geq 0$ and consider its binary representation as defined above.
    Then with $c_i \geq 2^{k_i-1}$, we get
    \begin{align*}
    \big(1+ \tfrac{1}{2^{\ell-1}}\big)\, c_i \geq c_i \geq c_i' = \sum_{j = k_i - \ell}^{k_i - 1} c_{ij}\, 2^j \geq c_i - 2^{k_i-\ell} = c_i - \frac{2^{k_i-1}}{2^{\ell-1}} \geq \big(1- \tfrac{1}{2^{\ell-1}}\big) c_i.
    \end{align*}
    The case of $c_i < 0$ can be shown in a similar way.\qed
\end{proof}

The concept of $\varepsilon$-optimality we use here slightly deviates from the conventional definition, which states that a point $x^\star$ is $\varepsilon$-optimal if $c\T x^\star \geq (1-\varepsilon)\, c\T x$ for all feasible $x$. As noted in \cite{orlin2008}, this traditional definition has several limitations, such as lacking translation invariance with respect to the variables and being not particularly meaningful when the optimal solution value is negative. Nevertheless, when $c \geq 0$, \cite[Lemma 4.1]{orlin2008} shows: If $x^\star$ is an $\varepsilon$-optimal solution with respect to the rounded objective~$c'$ and original objective~$c$ (Definition~\ref{def:epsilon_optimal}), then for all feasible $x \in \F$, the following holds:
\begin{align*}
  c\T x^\star \geq \frac{1}{1+\varepsilon} (c')\T x^\star \geq \frac{1}{1+\varepsilon} (c')\T x \geq \frac{1-\varepsilon}{1+\varepsilon} c\T x.
\end{align*}
Thus, we have $\frac{1-\varepsilon}{1+\varepsilon}$-optimality with respect to the traditional definition.

\section{Computational Experiments}
\label{sec:experiments}

This section evaluates the impact of $\ell$-bit rounding of objective coefficients on symmetry and solver performance. The study considers three problem classes: generated instances of the Capacitated Facility Location Problem, the Knapsack Problem, and a public test set of pseudo-Boolean instances with large objective coefficients.\footnote{Instance and generation scripts as well as all computational results can be found at \url{https://github.com/ZIB-IOL/BitRoundingAlmostSymmetries}.}

All experiments were performed on Intel(R) Xeon(R) Gold 6338 CPUs running at 2.00GHz with a time limit of one hour per instance and 64GB of RAM. SCIP was accessed via the PySCIPOpt package (version~5.1.1), which interfaces with SCIP~9.0.1.
The CP-SAT CP-Solver was accessed through the ortools Python package (version~9.14.6206). Both packages can be installed using \href{https://pypi.org/project/pip}{pip}. 
Obtaining reliable performance results for MIP solvers requires a carefully designed experimental setup, as even minor changes to the algorithm or input data can significantly affect solver behavior and performance. This phenomenon, commonly referred to as performance variability, is well-documented in the MIP literature~\cite{lodi2013performance}. To mitigate this effect, we solve each instance using five random seeds. 

We solve all instances with the original objective as well as the $\ell$-bit rounded version for $\ell=\{2, 3, 4, 5\}$ and compare the $\ell$-bit rounding approach to solving the original problem w.r.t.\ number of symmetry generators (\# Gen.), quality of the solution (\% Obj. Loss), solving time (Time [s]) and number of solved instances (\# Solved); the first three are reported as shifted geometric mean.
The objective loss is defined as $\abs{c\T x_{\ell}^{\star} - c\T x^{\star}} / \abs{c\T x^\star}$ for an optimal solution $x^\star$ of the original problem and an optimal solution $x^\star_\ell$ of the $\ell$-bit-rounded problem. We report \% Obj. Loss only for instances for which both versions are solved to optimality. 
Note that the number of generators is only a rough approximation of the `amount' of symmetry present in the instance, but since many symmetry handling techniques are based on the generators, this is useful information.

\subsection{Capacitated Facility Location Problem}

We consider the well-known \emph{capacitated facility location problems} (CFLP) as one of our test classes. Given $n$ facilities and $m$ customers, the problem is to decide which facilities to open in order to serve each customer while minimizing the total distance between facilities and the customers they serve. Formally, denote by $c_{ij}$ the distance from facility $i$ to customer $j$, $x_{ij} \in \{0,1\}$ the decision variable indicating whether facility $i$ serves customer $j$, $y_i$ the decision variable for opening facility $i$, $f_i$ the fixed cost associated with opening facility $i$, and $u_i$ the capacity of facility $i$, that is, the maximum number of customers it can serve. The problem can then be written as
\begin{align*}
    \min\; &\sum_{i=1}^n\sum_{j=1}^m c_{ij}\, x_{ij} + \sum_{i=1}^n f_i\, y_i \\
    \text{s.t.}\; &\sum_{i=1}^n x_{ij} = 1 && \forall j\in [m], \\
    &\sum_{j=1}^m x_{ij} \le u_i y_i && \forall i \in [n], \\
    &x_{ij} \in \{0, 1\} && \forall i\in [n],\; \forall j\in [m], \\
    &y_i \in \{0,1\} && \forall i\in [n].
\end{align*}

\textbf{Instance generation.}
In order for $\ell$-bit rounding to have an effect, we generated instances in which the customers are close together, such that rounding will make their distances to each facility equal. More precisely, we sample customer positions uniformly from a scaled unit square around the origin and facility positions are uniformly spaced around the scaled unit circle, i.e., at angles $\varphi_j = \frac{2\pi j}{n}$, $j \in [n]$. Capacities are randomly chosen between $\left\lceil\frac{m}{n}\right\rceil$ and $\left\lceil\frac{2m}{n}\right\rceil$ to ensure feasibility, while allowing for more assignment combinations. Per instance we assign each facility the same sampled capacity  $u_i = C$ for all $i \in [n]$. 

We generated 100 instances with $n \in [20, 500]$ facilities and $m \in [100, 2000]$ customers. During generation, we scale the distances $c_{ij}$ such that all coefficients lie in $[0, 1]$. We then apply varying levels of decimal point rounding, with higher levels leading to less symmetric coefficients.
Finally, all coefficients are scaled by $10^6$ to ensure they are large enough. 

\textbf{Results.}\label{sec:results_facloc}
Table~\ref{tab:cflp} presents the results for CFLP instances using $\ell$-bit rounding for both CP-SAT and SCIP, with and without symmetry handling enabled.
For CP-SAT, $\ell$-bit rounding substantially increases the number of symmetry generators, with lower values of $\ell$ generally leading to easier problems.
 Comparing the performance with and without symmetry detection reveals that symmetry handling becomes increasingly beneficial for rounded problems: enabling symmetry detection reduces runtime by \SI{23.5}{\percent} for the original objective, and this speedup grows to \SI{38.1}{\percent} for $\ell=4$ and
  \SI{48.6}{\percent} for $\ell=3$. Notably, for $\ell=2$, even without symmetry detection, the problem becomes significantly easier, achieving a \SI{90}{\percent} runtime reduction compared to the original, though symmetry handling still provides an additional \SI{40.8}{\percent} speedup. Overall, these results demonstrate that the CP solver benefits considerably from the higher
   symmetry induced by rounding, with a maximum loss of only \SI{5.01}{\percent} in solution quality.
   
\begin{table}[tb]
\scriptsize
\centering
\caption{CP-SAT and SCIP with and without symmetry on $5\times100$ CFLP instances.}
\label{tab:cflp}
\begin{tabular*}{\textwidth}{@{\extracolsep{\fill}}l l r r r r r r r@{}}
\toprule
\multirow{2}{*}{Solver} & \multirow{2}{*}{$\ell$-bit} 
& \multicolumn{3}{c}{No Symmetry} 
& \multicolumn{4}{c}{With Symmetry} \\
\cmidrule(){3-5} \cmidrule(){6-9}
& & \% Obj. Loss & Time [s] & \# Solved
&  \# Gen. & \% Obj. Loss & Time [s] & \# Solved \\
\midrule
& original & $\pm0.00$ & 505.71 & 240 & 2.42 & $\pm0.00$ & 386.56 & 254 \\
& $\ell=5$   & 3.41e${-1}$ & 423.38 & 256 & 26.36  & 3.44e${-1}$ & 313.40 & 279 \\
CP-SAT & $\ell=4$   & 7.10e${-1}$ & 410.49 & 258 & 96.55  & 7.29e${-1}$ & 254.22 & 281 \\
& $\ell=3$   & 3.45 & 276.93 & 299 & 209.78 & 3.49 & 142.35 & 317 \\
& $\ell=2$   & 4.90 & 52.69  & 451 & 186.01 & 5.01 & 31.18  & 480 \\
\midrule
& original & $\pm0.00$ & 392.38 & 253 & 2.63   & $\pm0.00$ & 398.38 & 263 \\
 & $\ell=5$   & 4.15e${-1}$ & 346.10 & 255 & 38.99  & 4.14e${-1}$ & 335.40 & 262 \\
SCIP & $\ell=4$   & 8.30e${-1}$ & 339.02 & 261 & 136.75 & 8.21e${-1}$ & 360.57 & 260 \\
& $\ell=3$   & 3.50 & 346.08 & 260 & 217.29 & 3.52 & 398.56 & 263 \\
 & $\ell=2$   & 4.94 & 135.93 & 337 & 343.63 & 4.97 & 171.98 & 347 \\
\bottomrule
\end{tabular*}
\end{table}

For SCIP, the picture is different. While the number of symmetry generators similarly increases for decreasing $\ell$ values, symmetry handling shows mixed effects on performance: enabling symmetry provides marginal benefits or even slowdowns for most roundings, though it does help solve a few additional instances. This indicates that for SCIP, the primary benefit of rounding comes from simplifying the problem structure rather than from symmetry exploitation.

\subsection{Knapsack Problem}

The second problem class is the Knapsack problem. Given $n$ items with values $c_i$ and weights $w_i$, $i \in [n]$, the problem is to choose a subset of items to maximize the sum of their values, while not exceeding a given knapsack capacity $W$.
Formally, the problem is
\begin{align*}
    \max\, \{c\T x \suchthat w\T x \leq W,\; x \in \{0, 1\}^n\}.
\end{align*}

\textbf{Instance generation.}
To generate instances, we distribute the $n$ items into $k$ clusters of size $s_j$, each with some base value $\gamma_j$ and base weight $\omega_j$, $j \in [k]$. We then randomly perturb the value of each item in the cluster, introducing asymmetry in values while maintaining the cluster structure. Specifically, for a fixed cluster $j \in [k]$ with base value $\gamma_j$, the items in that cluster have values $\gamma_j + \mathcal{U}(-\sigma, \sigma)$, where $\mathcal{U}(a,b)$ denotes the \emph{discrete} uniform distribution over $[a, b) \cap \Z$.
The total knapsack capacity is set to $W \define \smash{\frac{1}{2}\sum_{j=1}^k s_j \,\omega_j}$ to ensure a non-trivial solution.

Instances were generated with $n\in[50000, 250000]$ items and $k\in[50, 500]$ clusters with base values $\gamma_j \in \mathcal{U}(2^{10}, 2^{20})$ and noise $\sigma = 2^{12}$. We consider two different ranges for the weights, $\mathcal{U}(50,500)$ and $\mathcal{U}(50,10^5)$. For each combination of these parameters, we created one instance with balanced clusters, i.e., $s_j \approx \frac{n}{k}$, and another with random cluster sizes, such that in both cases $s_1 + \dots + s_k = n$, for a total of 100 instances of the knapsack problem.

\textbf{Results.}\label{sec:results_knapsack}
Table~\ref{tab:aggregated_lbit_scip_knapsack} shows our results for knapsack instances solved with SCIP.
We do not present results with CP-SAT on the knapsack instances, because preliminary experiments indicated that it currently does not perform well on these instances. One contributing factor may be that the large coefficients lead to very large SAT encodings. Experiments using the specialized knapsack solver of OR-Tools are left for future research.

The results follow a similar trend to the CFLP results for CP-SAT. With decreasing $\ell$, we see both a significant improvement in solving time and increase in solved instances. The number of symmetry generators also goes up for all $\ell$ compared to the original. In $13$ of the $100$ Knapsack instances, $\ell$-bit rounding produces more symmetry generators than the original formulation and for the remaining the number is equal. The dip for $\ell=2$ and $\ell=3$ compared to the higher $\ell$ comes from SCIP being able to fix all variables in presolving for one problem leading to no symmetry detection for the two lowest values of $\ell$. 
Turning off symmetry detection only notably influences the original problem formulation with \SI{14.3}{\percent} faster solving time on average and 10 additional solved instances. Analyzing this observation, without symmetry handling more problems are solved by fixing all variables after adding cuts found during the solving process. In comparison, with symmetry handling for the same problems SCIP fixes around \SI{98}{\percent} of the variables in similar time and then spends the remaining time trying to complete the dual proof.
As observed already for CFLP, also for the Knapsack instances SCIP benefits mainly from the simplified problem structure induced by $\ell$-bit rounding instead of increased symmetry.

\begin{table}[tb]
\scriptsize
\centering
\caption{SCIP with and without symmetry on $5\times100$ Knapsack instances.}
\label{tab:aggregated_lbit_scip_knapsack}
\begin{tabular*}{\textwidth}{@{\extracolsep{\fill}}l l r r r r r r r@{}}
\toprule
Solver & $\ell$-bit 
& \multicolumn{3}{c}{No Symmetry} 
& \multicolumn{4}{c}{With Symmetry} \\
\cmidrule(){3-5} \cmidrule(){6-9}
& & \% Obj. Loss & Time [s] & \# Solved
&  \# Gen. & \% Obj. Loss & Time [s] & \# Solved \\
\midrule
& original & $\pm0.00$ & 1280.95 & 160 & 410.34 & $\pm0.00$ & 1495.66 & 150 \\
& $\ell=5$ & 4.80e${-3}$ & 739.73 & 296 & 491.33 & 5.23e${-3}$ & 737.24 & 297 \\
SCIP & $\ell=4$ & 2.55e${-2}$ & 578.21 & 359 & 491.33 & 2.43e${-2}$ & 576.65 & 359 \\
& $\ell=3$   & 8.17e${-2}$ & 388.59 & 399 & 457.33 & 8.08e${-2}$ & 388.49 & 399 \\
 & $\ell=2$   & 2.88e${-1}$ & 203.36 & 449 & 462.39 & 2.88e${-1}$ & 208.17 & 449 \\
\bottomrule
\end{tabular*}
\end{table}

\subsection{Pseudo-Boolean Problems}
Lastly, we apply $\ell$-bit rounding to a test set of pseudo-Boolean (PB) instances. These instances are gathered from the well-established MIPLIB library \cite{miplib2017} translated into PB 0-1 format\footnote{The translated versions of the MIPLIB instances can be found at \url{ https://doi.org/10.5281/zenodo.3870965}} \cite{devriendt2021learn} and PB Competition instances used in the Pseudo-Boolean Competition 2024\footnote{Instances from the PB competition can be found at \url{https://www.cril.univ-artois.fr/PB24/details.html}}.
To ensure that $\ell$-bit rounding has a meaningful effect on objective coefficients, we filtered instances based on the size of their largest coefficients. Specifically, we kept instances whose largest coefficients are at least $2^{10}$. The final set consists of 72 instances from MIPLIB and 98 instances from the PB Competition, for a total of 170 instances.

\textbf{Results.}\label{sec:results_pb}
Table~\ref{tab:combined_lbit_pb} shows aggregated results for our set of PB instances solved with CP-SAT and SCIP, both with and without symmetry handling.
For CP-SAT, while we do not observe a direct effect on solving time, the solver produces a slightly larger number of symmetry generators for lower values of $\ell$, and the loss in solution quality remains low for any level of objective rounding. Furthermore, the total number of solved instances indicates that some problems that were unsolved with the original objective became easier with $\ell$-bit rounded coefficients, particularly for $\ell=2$ and $\ell=3$.
In contrast to CP-SAT on the CFLP instance class, here the effect of enabling symmetry is not as strong and for the $2$-bit rounded problems seemingly detrimental w.r.t. the number of solved instances.
For SCIP, $\ell$-bit rounding improves performance both with and without symmetry, with symmetry handling providing additional benefits on the rounded problems, up to \SI{17,8}{\percent} faster solving for $\ell=2$ while maintaining solution 
quality within \SI{0.76}{\percent} of the original optimum.
While the original problems are slightly faster without symmetry, applying $\ell$-bit rounding seems to allow the solver to utilize symmetries more efficiently and decrease solving time. 
Notably, across all $\ell$ the number of symmetry generators decreases compared to the original problem. This is similar to the behavior we observed for the Knapsack instances, with SCIP being able to fix more variables before symmetry computation, ultimately leading to a lower number of generators.

\begin{table}[tb]
\scriptsize
\centering
\caption{CP-SAT and SCIP with and without symmetry on $5\times170$ PB instances.}
\label{tab:combined_lbit_pb}
\begin{tabular*}{\textwidth}{@{\extracolsep{\fill}}l l r r r r r r r@{}}
\toprule
\multirow{2}{*}{Solver} & \multirow{2}{*}{$\ell$-bit} 
& \multicolumn{3}{c}{No Symmetry} 
& \multicolumn{4}{c}{With Symmetry} \\
\cmidrule(){3-5} \cmidrule(){6-9}
& & \% Obj. Loss & Time [s] & \# Solved
&  \# Gen. & \% Obj. Loss & Time [s] & \# Solved \\
\midrule
& original & $\pm0.00$ & 120.57 & 513 & 3.76 & $\pm0.00$ & 111.51 & 517 \\
& $\ell=5$   & 5.05e$-2$ & 120.37 & 513 & 4.61 & 5.23e$-2$ & 116.05 & 517 \\
CP-SAT & $\ell=4$   & 1.27e$-1$ & 121.28 & 515  & 4.84 & 1.28e$-1$ & 121.56 & 516 \\
& $\ell=3$   & 3.73e$-1$ & 119.55 & 538 & 5.10 & 3.83e$-1$ & 114.06 & 537 \\
& $\ell=2$   & 8.05e$-1$ & 113.03 & 541 & 5.21 & 8.25e$-1$ & 114.68 & 532 \\
\midrule
& original & $\pm0.00$ & 200.91 & 471 & 5.05 & $\pm0.00$ & 210.81 & 470 \\
 & $\ell=5$   & 5.14e$-2$ & 192.88 & 488 & 3.95 & 5.91e$-2$ & 186.23 & 484 \\
SCIP & $\ell=4$   & 1.38e$-1$ & 189.86 & 500 & 4.17 & 1.33e$-1$ & 183.29 & 499 \\
& $\ell=3$   & 3.38e$-1$ & 191.12 & 503 & 4.41 & 3.41e$-1$ & 180.94 & 498 \\
 & $\ell=2$   & 7.75e$-1$ & 177.96 & 499 & 4.63 & 7.63e$-1$ & 173.28 & 503 \\
\bottomrule
\end{tabular*}
\end{table}

\section{Conclusion}
\label{sec:conclusion}

The experiments of this article clearly show a significant speed-up in solving time when decreasing the number of significant bits in the objective function, while the relative error to the original optimal value is small.
Moreover, on the considered instances, we see a clear trend to more symmetry generators indicating almost symmetry.
Often, symmetry handling methods are able to exploit this to boost performance, but this also depends on the type of instance and solver, opening opportunities for future research.

%
%
%
%
\bibliographystyle{splncs04}
\bibliography{roundsym}

\begin{thebibliography}{10}
\providecommand{\url}[1]{\texttt{#1}}
\providecommand{\urlprefix}{URL }
\providecommand{\doi}[1]{https://doi.org/#1}

\bibitem{BestuzhevaEtAl23}
Bestuzheva, K., Besan\c{c}on, M., Chen, W.K., Chmiela, A., Donkiewicz, T., van Doornmalen, J., Eifler, L., Gaul, O., Gamrath, G., Gleixner, A., Gottwald, L., Graczyk, C., Halbig, K., Hoen, A., Hojny, C., van~der Hulst, R., Koch, T., L\"{u}bbecke, M., Maher, S.J., Matter, F., M\"{u}hmer, E., M\"{u}ller, B., Pfetsch, M.E., Rehfeldt, D., Schlein, S., Schl\"{o}sser, F., Serrano, F., Shinano, Y., Sofranac, B., Turner, M., Vigerske, S., Wegscheider, F., Wellner, P., Weninger, D., Witzig, J.: Enabling research through the {SCIP Optimization Suite} 8.0. ACM Trans. Math. Softw.  \textbf{49}(2) (2023). \doi{10.1145/3585516}, article 22

\bibitem{BolusaniEtal2024OO}
Bolusani, S., Besan{\c{c}}on, M., Bestuzheva, K., Chmiela, A., Dion{\'{i}}sio, J., Donkiewicz, T., van Doornmalen, J., Eifler, L., Ghannam, M., Gleixner, A., Graczyk, C., Halbig, K., Hedtke, I., Hoen, A., Hojny, C., van~der Hulst, R., Kamp, D., Koch, T., Kofler, K., Lentz, J., Manns, J., Mexi, G., M\"{u}hmer, E., Pfetsch, M.E., Schl{\"o}sser, F., Serrano, F., Shinano, Y., Turner, M., Vigerske, S., Weninger, D., Xu, L.: {The SCIP Optimization Suite 9.0}. Technical report, Optimization Online (2024), \url{https://optimization-online.org/2024/02/the-scip-optimization-suite-9-0/}

\bibitem{devriendt2021learn}
Devriendt, J., Gleixner, A., Nordstr{\"o}m, J.: Learn to relax: Integrating 0-1 integer linear programming with pseudo-{B}oolean conflict-driven search. Constraints  \textbf{26}(1),  26--55 (2021). \doi{10.1007/s10601-020-09318-x}

\bibitem{DonG05}
Donaldson, A.F., Gregory, P. (eds.): Proceedings of the Almost-Symmetry in Search (SymNet) Workshop, New Lanark, Scotland (2005), technical report: TR-2005-201, Department of Computing Science University of Glasgow, \url{https://www.doc.ic.ac.uk/~afd/homepages/edited_volumes/SymNet2005.pdf}

\bibitem{GhoS11}
Ghoniem, A., Sherali, H.D.: Defeating symmetry in combinatorial optimization via objective perturbations and hierarchical constraints. IIE Transactions  \textbf{43},  575--588 (2011). \doi{10.1080/0740817X.2010.541899}

\bibitem{miplib2017}
Gleixner, A., Hendel, G., Gamrath, G., Achterberg, T., Bastubbe, M., Berthold, T., Christophel, P.M., Jarck, K., Koch, T., Linderoth, J., Lübbecke, M., Mittelmann, H.D., Ozyurt, D., Ralphs, T.K., Salvagnin, D., Shinano, Y.: {MIPLIB 2017: Data-Driven Compilation of the 6th Mixed-Integer Programming Library}. Mathematical Programming Computation  (2021). \doi{10.1007/s12532-020-00194-3}, \url{https://doi.org/10.1007/s12532-020-00194-3}

\bibitem{KnuOP18}
Knueven, B., Ostrowski, J., Pokutta, S.: Detecting almost symmetries of graphs. Math. Prog. Comp.  \textbf{10},  143--185 (2018). \doi{10.1007/s12532-017-0124-3}

\bibitem{KorV18}
Korte, B., Vygen, J.: Combinatorial Optimization. Theory and Algorithms, Algorithms and Combinatorics, vol.~21. Springer, Heidelberg, 6th edn. (2018). \doi{10.1007/978-3-662-56039-6}

\bibitem{BPPP2015}
Le~Bodic, P., Pavelka, J., Pfetsch, M., Pokutta, S.: Solving {MIPs} via scaling-based augmentation. Discrete Optimization  \textbf{27},  1--25 (2018). \doi{doi: 10.1016/j.disopt.2017.08.004}

\bibitem{lodi2013performance}
Lodi, A., Tramontani, A.: Performance variability in mixed-integer programming. In: Theory driven by influential applications, pp. 1--12. INFORMS (2013). \doi{10.1287/educ.2013.0112}

\bibitem{Mar10}
Margot, F.: Symmetry in integer linear programming. In: J{\"u}nger, M., Liebling, T.M., Naddef, D., Nemhauser, G.L., Pulleyblank, W.R., Reinelt, G., Rinaldi, G., Wolsey, L.A. (eds.) 50 Years of Integer Programming. pp. 647--686. Springer (2010)

\bibitem{orlin2008}
Orlin, J.B., Schulz, A.S., Sengupta, S.: $\varepsilon$-optimization schemes and $l$-bit precision: Alternative perspectives for solving combinatorial optimization problems. Discrete Optimization  \textbf{5}(2),  550--561 (2008). \doi{10.1016/j.disopt.2007.08.004}

\bibitem{ortools_cp}
Perron, L., Didier, F.: Google {OR-Tools -- CP-SAT} (2025), \url{https://developers.google.com/optimization/cp/cp_solver}

\bibitem{PfeR19}
Pfetsch, M.E., Rehn, T.: A computational comparison of symmetry handling methods for mixed integer programs. Mathematical Programming Computation  \textbf{11}(1),  37--93 (2019). \doi{10.1007/s12532-018-0140-y}

\bibitem{schulz2002complexity}
Schulz, A.S., Weismantel, R.: The complexity of generic primal algorithms for solving general integer programs. Mathematics of Operations Research  \textbf{27}(4),  681--692 (2002). \doi{10.1287/moor.27.4.681.305}

\bibitem{schulz19950}
Schulz, A.S., Weismantel, R., Ziegler, G.M.: 0/1-integer programming: Optimization and augmentation are equivalent. In: Algorithms -- {ESA} '95, Proceedings. pp. 473--483 (1995). \doi{10.1007/3-540-60313-1_164}

\end{thebibliography}

\end{document}